\documentclass{article}
\usepackage{graphicx,amsmath,amsthm, amssymb} 
\newtheorem{proposition}{Proposition}[section]
\newtheorem{thm}[proposition]{Theorem}

\newtheorem{lemma}[proposition]{Lemma}
\newtheorem{defn}[proposition]{Definition}

\newcommand{\otp}{\mathrm{otp}}
\newcommand{\forces}{\Vdash}
\newcommand{\restrict}{\upharpoonright}
\newcommand{\NS}{{\sf NS}}
\newcommand{\CFB}{{\sf CFB}}

\newcommand{\bbP}{\mathbb{P}}
\newcommand{\bbQ}{\mathbb{Q}}

\newcommand{\cP}{\mathcal{P}}

\newcommand{\bls}{\vspace{\baselineskip}}

\newtheorem{remark}[proposition]{Remark}

\title{{\sf PFA} and the definability of the nonstationary ideal}
\author{Stefan Hoffelner\thanks{The first and third authors are funded by the Deutsche Forschungsgemeinschaft (DFG, German Research
		Foundation) under Germany’s Excellence Strategy EXC 2044 390685587, Mathematics M\"unster:
		Dynamics - Geometry - Structure.} \and Paul Larson \and Ralf Schindler
		\and Liuzhen Wu}
\date{April 2023}

\begin{document}

\maketitle

\begin{abstract}
We produce, relative to a ${\sf ZFC}$ model with a supercompact cardinal, a ${\sf ZFC}$ model of the Proper Forcing Axiom in which the nonstationary ideal on $\omega_1$ is $\Pi_1$-definable in a parameter from $H_{\aleph_2}$.     
\end{abstract}

\section{Introduction}

A subset of $\omega_1$ is said to be nonstationary if there exists a club $C \subseteq \omega_1$ disjoint from it. It follows that the ideal ${\sf NS}_{\omega_1}$ of nonstationary subsets of $\omega_{1}$ is $\Sigma_1$-definable in the parameter $\omega_1$. Theorem 1.3 of \cite{HLSW} shows that in the presence of ${\sf BPFA}$ (the Bounded Proper Forcing Axiom) ${\sf NS}_{\omega_1}$ may also be $\Pi_1$-definable in the parameter $\omega_1$. On the other hand, Martin's Maximum and Woodin's axiom (*) each imply that ${\sf NS}_{\omega_1}$ cannot be $\Pi_1$-definable over $H_{\aleph_{2}}$ in any parameter from $H_{\aleph_2}$, cf.\ \cite[Theorem 2.3]{HLSW} and \cite{SS}.
The current paper strengthens Theorem 1.3 of \cite{HLSW} by proving the following theorem. 


\begin{thm}\label{mainthrm}
	If there exists a supercompact cardinal, then there exists a proper forcing extension in which ${\sf PFA}$ holds and 
	${\sf NS}_{\omega_{1}}$ is $\Pi_{1}$-definable in a parameter from $H_{\aleph_{2}}$.
\end{thm}

Theorem \ref{mainthrm} follows from Theorem \ref{mainthrmdetail} below. The parameter cannot be removed from Theorem \ref{mainthrm}, as Corollary 4.13 of \cite{LSS} shows that under {\sf PFA}, ${\sf NS}_{\omega_1}$ is not $\Pi_1$-definable in the parameter $\omega_1$.

The overall strategy of the proof is as follows. First the parameter $A$ (a partition of $\omega_{1}$ into $\aleph_{1}$-many pieces) is added generically by countable approximations. This is followed by a countable support iteration which interleaves a certification forcing at successor stages with the standard ${\sf PFA}$ iteration at limit stages (due to Baumgartner; see \cite{FMS} for the corresponding iteration for Martin's Maximum). The entire iteration (including the first step adding $A$) is proper, although the tails of the iteration are only semi-proper over the nontrivial  initial extensions. One potentially novel aspect of the construction is that the bookkeeping used in the certification part of the iteration is chosen with some care, in order to enable the construction of suitably generic conditions. The certification mechanism used is not tied closely to the nonstationary ideal, and could be used to code other ideals. 

In the model produced, ${\sf NS}_{\omega_{1}}$ is not saturated (and in fact Canonical Function Bounding fails, see Theorem \ref{mainthrmdetail}). It remains open whether ${\sf PFA}$ plus the saturation of ${\sf NS}_{\omega_{1}}$ (or Canonical Function Bounding) is consistent with the $\Pi_{1}$-definability of ${\sf NS}_{\omega_{1}}$ relative to a subset of $\omega_{1}$.

\section{A coding machinery}

We write $\otp(a)$ for the ordertype of a set $a$ of ordinals. Given a set $A \subseteq \omega_{1}$, $\tilde{A}$ (``$A$-tilde") is the set of $\gamma \in [\omega_{1}, \omega_{2})$ for which there is a bijection $b \colon \omega_{1} \to \gamma$ such that $\{ \alpha < \omega_{1} : \otp(b[\alpha]) \in B\}$ contains a club (see \cite{hugh, size}).

Let $A$ be a subset of $\omega_{1}$, and let $\gamma \geq \omega_{1}$ be an ordinal. We let $Q(A,\gamma)$ denote the natural partial order to force $\gamma$ into $\tilde{A}$. 
That is, conditions in $Q(A, \gamma)$ are sequences \[p = \langle a_{\alpha} : \alpha \leq \zeta \rangle,\] for some countable ordinal $\zeta$, such that 
\begin{itemize}
	\item for all $\alpha \leq \zeta$, $a_{\alpha} \in [\gamma]^\omega$ and $\otp(a_{\alpha}) \in S$,
	\item $a_{\alpha} \subsetneq a_{\beta}$ when $\alpha < \beta \leq \zeta$ and 
	\item $a_{\beta} = \bigcup_{\alpha < \beta} a_{\alpha}$ when $\beta \leq \zeta$ is a limit ordinal. 
\end{itemize}
The order on $Q(A, \gamma)$ is defined by setting $p \leq q$ to hold if $p$ end-extends $q$.

We will use partial orders of the form $Q(A, \gamma)$ only the in the case where $\gamma$ is a measurable cardinal. In doing so, we will be using the following standard fact, which appears in many places, including Lemma 1.1.21 of \cite{stat-tower-book}. We include a proof for the convenience of the reader. 

\begin{lemma}\label{endexlem} Suppose that $\kappa$ is a measurable cardinal, $\theta > 2^{\kappa}$ is a regular cardinal, $A \subseteq \omega_{1}$ is stationary and $X \prec H_{\theta}$ is countable with $\kappa \in X$. Then there exists a countable $Y \prec H_{\theta}$ such that
	\begin{itemize}
		\item $X \subseteq Y$, 
		\item $X \cap \omega_{1} = Y \cap \omega_{1}$ and 
		\item $\otp(Y \cap \kappa) \in A$. 
	\end{itemize}
\end{lemma}

\begin{proof}
	Let $U \in X$ be a normal measure on $\kappa$. For any countable $Z \prec H_{\theta}$ with $U \in Z$, if $$\eta \in \bigcap (U \cap Z)$$ and if $Z'$ is the set of values $f(\eta)$ for $f$ a function in $Z$ with domain $\kappa$, then $Z' \prec H_{\theta}$, $Z \subseteq Z'$ and 
	$Z' \cap \kappa$ end-extends $Z \cap \kappa$, i.e., $Z \cap \kappa = Z' \cap {\rm sup}(Z\cap \kappa)$. 
	
	We may then form a continuous chain $\langle X_{\alpha} \colon \alpha <\omega_1\rangle$ of countable elementary substructures of $H_\theta$ such that $X_0 = X$ and $X_{\alpha} \cap \kappa = X_{\beta} \cap {\rm sup}(X_{\alpha}\cap \kappa)$ for all $\alpha \leq \beta < \omega_{1}$. It follows that $\{ \otp(X_{\alpha} \cap \kappa) : \alpha < \omega_{1}\}$ is a club subset of $\omega_{1}$. As $A$ is stationary, there is then some $\alpha <\omega_1$ such that setting $Y = X_{\alpha}$, $Y$ is as desired. 
\end{proof} 


Our first application of Lemma \ref{endexlem} is the following. 

\begin{lemma}\label{semi-proper}
	Let $A$ be a stationary subset of $\omega_{1}$ and let $\kappa$ be a measurable cardinal.  Then $Q(A, \kappa)$ is semi-proper and forces the statement ``$\kappa \in \tilde{A}$". 
\end{lemma}

\begin{proof}
	To see that $Q(A, \kappa)$ is semi-proper, let $\theta > 2^{\kappa}$ be a regular cardinal and let $X \prec H_{\theta}$ be countable with $A, \kappa \in X$. Applying Lemma \ref{endexlem}, let $Y \prec H_{\theta}$ contain $X$, with $X \cap \omega_{1} = Y \cap \omega_{1}$, and $\otp(Y \cap \kappa) \in A$. Then the union of any $(Q(A, \kappa), Y)$-generic filter is a condition in $Q(A, \kappa)$. That $Q(A, \kappa)$ forces $\kappa$ into $\tilde{A}$ follows by a standard genericity argument. 
\end{proof}

Let us say that ${\vec A}= \langle A_{\alpha} \colon \alpha <\omega_1 \rangle$ {\em splits} $\omega_1$ {\em into stationary sets} if
\begin{itemize}
	\item each $A_{\alpha}$ is a stationary subset of $\omega_1$,
	\item $\omega_1 = \bigcup_{\alpha<\omega_1} A_\alpha$, and
	\item $A_\alpha \cap A_{\beta} = \emptyset$ for all $\alpha < \beta < \omega_{1}$. 
\end{itemize}
Given such a $\vec{A}$, and a set $S \subseteq \omega_{1}$, we will write $\vec{A}(S)$ for $\bigcup_{\alpha \in S}A_{\alpha}$. Given an ordinal $\gamma$, we will say that  $S$ is {\em certified at} $\gamma$ ({\em modulo} ${\vec A}$) if $\gamma$ is in the tilde of $\vec{A}(S)$, i.e., if there exists a sequence $\langle a_{\alpha} \colon \alpha < \omega_1 \rangle$ such that
\begin{itemize}
	\item for all $\alpha < \omega_{1}$ $a_{\alpha} \in [\gamma]^\omega$ and $\otp(a_{\alpha}) \in S$,
	\item $a_{\alpha} \subsetneq a_{\beta}$ for all $\alpha < \beta < \omega_{1}$,
	\item $a_{\beta} = \bigcup_{\alpha < \beta} a_{\alpha}$ when $\beta <\omega_1$ is a limit ordinal, and 
	\item $\gamma = \bigcup_{\alpha <\omega_1} a_{\alpha}$. 
\end{itemize}
Note that if $S$ is certified at $\gamma$ then so is every superset of $S$. 

Given an ordinal $\gamma$ in the interval $[\omega_{1}, \omega_{2})$, a \emph{canonical function} for $\gamma$ is a function $f \colon \omega_{1} \to \omega_{1}$ such that, for some bijection $\pi \colon \omega_{1} \to \gamma$, $f(\alpha)$ is the ordertype of $\pi[\alpha]$ for each $\alpha < \omega_{1}$. Fixing such an $f$, and $S$ and $\vec{A}$ as above, we say that  $S$ is {\em coded at} $\gamma$ ({\em modulo} ${\vec A}$) if for all $\alpha<\omega_1$,
$$\alpha \in S \Leftrightarrow \{ \beta <\omega_1 \colon f(\beta) \in A_\alpha \} \mbox{ is stationary.}$$
Since any two canonical functions for $\gamma$ agree on a club subset of $\omega_{1}$, this definition does not depend on the choice of $f$. 

\emph{Canonical Function Bounding} ($\sf{CFB}$) is the statement that for each function $f \colon \omega_{1}\to \omega_{1}$ there is a canonical function $g \colon \omega_{1} \to \omega_{1}$ for some $\gamma \in [\omega_{1},  \omega_{2})$ such that the set $\{ \alpha  <\omega_{1} : f(\alpha) < g(\alpha)\}$ contains a club. It is a standard fact that ${\sf CFB}$ follows from the saturation ${\sf NS}_{\omega_{1}}$ (the proof starts by noting that if ${\NS}_{\omega_{1}}$ is saturated, then forcing with $\cP(\omega_{1})/{\NS_{\omega_{1}}}$ cannot collapse $\omega_{2}$, which means that the identity function on $\omega_{1}$ must represent $\omega_{2}^{V}$ in any induced generic ultrapower).

\begin{lemma}\label{semi-proper2}
	Suppose that $\vec{A}= \langle A_{\alpha} : \alpha < \omega_{1} \rangle$ splits $\omega_{1}$ into stationary sets, and that $S \subseteq \omega_{1}$ is nonempty. Let $\kappa$ be a measurable cardinal.
	Suppose that $\vec{a} = \langle a_\alpha \colon \alpha<\omega_1 \rangle$ is $Q(\vec{A}(S), \kappa)$-generic over $V$. Then, in $V[\vec{a}]$, 
	\begin{enumerate}
		
		\item $S$ is certified at $\kappa$ via $\vec{A}$, and 
		
		\item $S$ is coded at $\kappa$, modulo $\vec{A}$.
	\end{enumerate}
\end{lemma}

\begin{proof}
The first part of the lemma follows from the second part of Lemma \ref{semi-proper}. The second follows similarly, by a genericity argument, as follows. The function $f \colon \omega_{1} \to \omega_{1}$ defined by letting $f(\beta)$ be $\otp(a_{\beta})$ is a canonical function for $\kappa$. 
If $\alpha \in \omega_{1} \setminus S$, then $\{ \beta <\omega_1 \colon {\rm otp}(a_{\beta}) \in A_\alpha \} = \emptyset$. We thus only need to see that if $\alpha \in S$, then $\{ \beta <\omega_1 \colon {\rm otp}(a_{\beta}) \in A_\alpha \}$ is stationary in $V[g]$.

Let $p \in {\mathbb P}$ and ${\dot C} \in V^{\mathbb P}$ be such that $p \Vdash {\dot C}$ is a club subset of $\omega_1$. Let $\theta > 2^{\kappa}$ be a regular cardinal and let $X$ be a countable elementary submodel of $H(\theta)$ with $p$ and $\dot{C}$ as members. Applying Lemma \ref{endexlem}, let $Y \prec H_{\theta}$ contain $X$, with $X \cap \omega_{1} = Y \cap \omega_{1}$, and $\otp(Y \cap \kappa) \in A_{\alpha}$. Then the union of any $(Q(\vec{A}(S), \kappa), Y)$-generic filter is a condition in $Q(\vec{A}(S), \kappa)$ forcing that $X \cap \omega_{1} \in \dot{C}$ and $a_{X \cap \omega_{1}} = Y \cap \kappa$, so $\otp(a_{X \cap \omega_{1}}) \in A_{\alpha}$.
\end{proof}

Note that being certified via $\vec{A}$ at a given ordinal is $\Sigma_1$ in $\omega_{1}$, so absolute to outer models. The property of being coded modulo $\vec{A}$ is $\Sigma_1$ in $\NS_{\omega_{1}}$ and $\vec{A}$, and therefore absolute to models preserving stationary subsets of $\omega_{1}$.

In what follows, we write ${\rm Col}(\omega_1,\omega_1)$ for the partial order consisting of all functions $p \colon \zeta \rightarrow \omega_1$, for some countable ordinal $\zeta$, ordered by end-extension. If $g$ is ${\rm Col}(\omega_1,\omega_1)$-generic over $V$, then we confuse $g$ with $\bigcup g$, a function from $\omega_1$ to $\omega_1$, and also write it as $g$. 

Given a function $g \colon \omega_{1} \to \omega_{1}$, we let the \emph{partition of} $\omega_{1}$ \emph{induced by} $g$ be the sequence $\langle A_{\alpha} : \alpha < \omega\rangle$ such that each $A_{\alpha}$ is the set $\{ \beta < \omega_{1} : g(\beta)= \alpha\}$. 
The proof of the following standard fact is elementary. 

\begin{lemma}\label{folklore}
If $g \colon \omega_1 \to \omega_1$ is ${\rm Col}(\omega_1,\omega_1)$-generic over $V$ then 
in $V[g]$ the partition of $\omega_{1}$ induced by $g$ 
splits $\omega_1$ into stationary sets. 
\end{lemma}

In the rest of the paper we will write $P(S, \gamma, g)$ for the partial order $Q(\vec{A}(S), \gamma)$, where $S$ is a subset of $\omega_{1}$, $\gamma \geq \omega_{1}$ is an ordinal, and $\vec{A}$ is the partition of $\omega_{1}$ induced by $g$. 



\section{The forcing iteration}

We are now ready to define our forcing iteration. In order to facilitate the discussion, let us introduce an ad hoc term for the kind of iterations which we will be interested in.

\begin{defn}\label{appropriate}
Let $\langle {\mathbb P}_\eta , {\dot {\mathbb Q}}_\xi \colon \eta \leq \delta , \xi < \delta\rangle$ be a countable support iteration of forcings. We call this iteration {\em appropriate} if ${\mathbb P}_0 = {\rm Col}(\omega_1,\omega_1)$ and there exists a sequence \[\langle \dot{S}_{\xi}, \alpha_{\xi}, \kappa_{\xi} : \xi < \delta \rangle\] such that, for all $\xi > 0$, either 
\begin{enumerate}
\item $\dot{\mathbb Q}_\xi$ is a $\bbP_{\xi}$-name for a proper forcing
or
\item $\dot{S}_{\xi}$ is a $\bbP_{\xi}$-name for a stationary subset of $\omega_{1}$ with $\alpha_{\xi}$ as a member, $\kappa_{\xi}$ is a measurable cardinal greater than $|\bbP_{\xi}|$ and, letting $\dot{g}_{0}$ be a $\bbP_{0}$-name for the generic function from $\omega_{1}$ to $\omega_{1}$ added by $\bbP_{0}$, 
$\Vdash_{{\mathbb P}_\xi} {\dot {\mathbb Q}}_\xi = P(\dot{S}_{\xi},\check{\kappa}_{\xi}, \dot{g}_0).$
\end{enumerate}
\end{defn}


In what follows we let $\dot{g}_{\xi}$ (for some ordinal $\xi$, relative to an appropriate iteration) denote the canonical name for the generic filter for $\bbP_{\xi}$, and when talking of a particular generic filter $g \subseteq \bbP_{\delta}$, let $g_{\xi}$ denote the restriction of $g$ to $\bbP_{\xi}$. 

\begin{lemma}\label{lemma}
Let $\bar{\bbP}= \langle {\mathbb P}_\eta , {\dot {\mathbb Q}}_\xi \colon \eta \leq \delta , \xi < \delta\rangle$ be an appropriate iteration, as witnessed by $W = \langle \dot{S}_{\xi}, \alpha_{\xi}, \kappa_{\xi} : \xi < \beta \rangle$. 
Let $\rho < \delta$ be nonzero, let $\theta > 2^{|\bbP|}$ be a regular cardinal and let $X$ be a countable elementary substructure of $H_\theta$ with $\bar{\bbP}$, $W$ and $\rho$ in $X$.
Suppose that $q$, ${\dot p}$ are such that
\begin{enumerate}
\item $q \in {\mathbb P}_\rho$ is $(X,{\mathbb P}_\rho)$-generic, 
\item for each $\xi \in X \cap \delta$, $q(0)(\otp(X \cap \kappa_{\xi})) = \alpha_{\xi}$, 
\item  ${\dot p} \in V^{{\mathbb P}_\rho}$, and 
\item  $q \Vdash_{{\mathbb P}_\rho} {\dot p} \in {\mathbb P}_\delta \cap X \wedge 
{\dot p} \upharpoonright \rho \in {\dot G}_\rho$. 
\end{enumerate}
Then there is a condition $r \in \bbP_{\delta}$ such that
\begin{itemize}
\item $r$ is $(X,{\mathbb P}_\delta)$-generic,
\item $r \upharpoonright \rho = q$, and
\item $r \Vdash_{{\mathbb P}_\delta} {\dot p} \in {\dot G}_\delta$.
\end{itemize}
\end{lemma}

\begin{proof}
	The proof is by induction on $\delta$.
If $\delta$ is a limit ordinal, then this is by the usual proper forcing argument, see e.g.\ the proof of \cite[Lemma 31.17]{jech}. If $\delta$ is a successor, then we may assume
that $\delta=\rho+1$. If $\Vdash_{{\mathbb P}_\rho} {\dot {\mathbb Q}}_\rho$ is proper, then this is again by the usual proper forcing argument, see e.g.\ \cite[Lemma 31.18]{jech}. Let us thus assume that
$\dot{S}_{\rho}$, $\alpha_{\rho}$ and $\kappa_{\rho}$ are as in the second case of Definition \ref{appropriate}. 

Suppose now that $g_{\rho}$ is ${\mathbb P}_{\rho}$-generic over $V$ with $q \in g_{\rho}$. We have that 
\begin{itemize}
	\item ${\dot {\mathbb Q}}_{\rho, g_{\rho}}= P(\dot{S}_{\rho, g_{\rho}}, \kappa_{\alpha}, g_0)$; 
	\item $\dot{S}_{\rho, g_{\rho}}$ is in $X[g_{\rho}]$ and stationary in $V[g_{\rho}]$; 
	\item $\kappa_{\rho}$ is in $X$ and measurable in $V[g_{\rho}]$;  
	\item $g_0({\rm otp}(X \cap \kappa_{\rho})) = \alpha_{\rho}$ and $\alpha_{\rho} \in \dot{S}_{\rho, g_{\rho}}$; 
	\item ${\dot p}_g \in X[g_{\rho}]$.
\end{itemize}
We may then produce in a standard fashion, in much the same way as in the proof of Lemma \ref{semi-proper}, some $s \in P(\dot{S}_{\rho, g_{\rho}}, \kappa_{\rho}, g_0)$ such that 
\begin{itemize}
\item $s <_{P(\dot{S}_{\rho, g_{\rho}}, \kappa_{\rho}, g_0)} {\dot p}_{g_{\rho}}(\rho)$,
\item ${\rm dom}(s) = (X \cap \omega_1) +1= (X[g] \cap \omega_1) +1$,
\item for each $D \in X[g_{\rho}]$ which is dense in $P(\dot{S}_{\rho, g_{\rho}}, \kappa_{\rho}, g_0)$ there is some \[{\bar s} >_{P(\dot{S}_{\rho, g_{\rho}}, \kappa_{\rho}, g_0)} s\] with ${\bar s} \in D \cap X[g]$, and 
\item $s(X \cap \omega_1)=X \cap \kappa_{\rho}$.
\end{itemize}
In particular, $s$ is $(X[g_{\rho}],P(\dot{S}_{\rho, g_{\rho}}, \kappa_{\rho}, g_0))$-generic.

By fullness, there is then some ${\mathbb P}_\rho$-name ${\dot s}$  such that $q$ forces that ${\dot s} \in {\dot {\mathbb Q}}_{\rho}$ is $(X[{\dot g}_{\rho}],{\dot {\mathbb Q}}_{\rho})$-generic and ${\dot s} <_{\bbQ_{\rho}} {\dot p}(\rho)$. It follows that $r = q^{\frown} {\dot s}$ is as desired.
\end{proof}

Lemma \ref{lemma} (plus the ability to choose $(X, \bbP_{0})$-generic $q$ satisfying the hypothesis of the lemma in case $\rho = 0$ with respect to any given $\dot{p}$) gives the following. 

\begin{thm}\label{mainthrmdetail0}
	If Let $\langle {\mathbb P}_\eta , {\dot {\mathbb Q}}_\xi \colon \eta \leq \delta ,\, \xi < \delta\rangle$ is an appropriate iteration then for every $\eta \leq \delta$,  ${\mathbb P}_\eta$ is proper.
\end{thm}

\begin{remark}\label{sprem}\normalfont
	
	Since appropriate iterations are proper, they preserve the property of having uncountable cofinality. It follows that the tails of appropriate iterations also preserve uncountable cofinality. Each iterand in an appropriate iteration is semi proper (this follows from Lemma \ref{semi-proper} in the second case). A Revised Countable Support  (RCS) iteration of semi-proper forcings is semi-proper (\cite{Fuchs, Miyamoto, Shelah}). However, an RCS iteration of partial orders is equivalent to the corresponding countable support iteration if it preserves uncountable cofinalities. It follows then that the tails of appropriate iterations are semi-proper, and in particular preserve stationary subsets of $\omega_{1}$. 

\end{remark}

To complete the proof of the main theorem, suppose that $\delta$ is a supercompact cardinal. By a {\em Laver function} for $\delta$ we mean some $R \colon \delta \rightarrow V_\delta$ such that for all $X \in V$ there are ${\bar \delta} < {\bar \theta} < \delta$ and $\theta > \delta$ together with an elementary embedding $$j \colon H_{\bar \theta} \rightarrow H_\theta$$ such that
\begin{itemize}
\item ${\rm crit}(j) = {\bar \delta}$,
\item $j({\bar \delta}) = \delta$,
\item $X \in H_\theta$, and 
\item $j(R({\bar \delta})) = X$.
\end{itemize}

As every supercompact cardinal has a Laver function \cite{Laver}, we fix a Laver function $R$ for $\delta$. We also fix a function $\pi \colon \delta \to \delta \times \omega_{1}$, with component functions $\pi_{0}$ and $\pi_{1}$, such that 
\begin{itemize}
	\item $\pi_{1}(\xi) \leq \xi$ for all $\xi < \delta$, and 
\item for each pair $(\eta, \alpha) \in \delta \times\omega_{1}$, $\pi^{-1}[\{ (\eta, \alpha)\}]$ contains $\delta$ many successor ordinals. 
\end{itemize}

We then define an appropriate iteration 
\[\langle {\mathbb P}_\eta , {\dot {\mathbb Q}}_\xi \colon \eta \leq \delta , \xi < \delta\rangle\] of length $\delta+1$ having properties (1)-(3) below. In doing so we fix for each $\eta < \delta$ a wellordering $\leq_{\eta}$ of the set $N_{\eta}$ consisting of the nice $\bbP_{\eta}$-names for stationary subsets of $\omega_{1}$ (nice in the sense of \cite{Kunen}; there will be less than $\delta$-many such names, and each stationary subset of $\omega_{1}$ in the $\bbP_{\eta}$-extension will be the corresponding realization of one of them).
Given $(\eta, \alpha) \in \delta \times \omega_{1}$, let $N_{\eta, \alpha}$ be the set of $\dot{S} \in N_{\eta}$ such that $\forces_{\bbP_{\eta}} \check{\alpha} \in \dot{S}$. Note then that if $g_{\eta}$ is $V$-generic for $\bbP_{\eta}$, $S \in V[g_{\eta}]$ is a stationary subset of $\omega_{1}$ and $\alpha \in S$, then $S$ is the realization via $g_{\eta}$ of some element of $N_{\eta, \alpha}$. 

Now we fix the following iteration. 
\begin{enumerate}
\item We let ${\mathbb P}_0$ be ${\rm Col}(\omega_1,\omega_1)$. 
\item If $\xi \in (0, \delta)$ is a limit ordinal, then ${\dot {\mathbb Q}}_\xi = R(\xi)$, provided that $\Vdash_{{\mathbb P}_\xi} ``R(\xi)$ is a proper forcing";
${\dot {\mathbb Q}}_\xi$ is trivial otherwise.  
\item If $\xi < \delta$ is a successor ordinal, and $\kappa$ is the least measurable cardinal strictly above $|{\mathbb P}_\xi|$, then ${\dot {\mathbb Q}}_\xi = P(\dot{S}, \check{\kappa}, \dot{g}_{0})$
and $\dot{S}$ is the $\leq_{\pi_{0}(\xi)}$-least $\bbP_{\pi_{0}(\xi)}$-name such that
\begin{itemize}
	\item $\forces_{\bbP_{\pi_{0}(\xi)}} \pi_{1}(\xi) \in \dot{S}$ and 
	\item 
	the realization of $\dot{S}$ by $g_{\pi_{0}(\xi)}$ is stationary and not certified 
	at any member of $\omega_{2}^{V[g_{\xi}]}$ via the partition of $\omega_{1}$ induced by the realization of $\dot{g}_{0}$, 
\end{itemize}
if such a $\dot{S}$ exists; otherwise we let $\dot{S}$ be $\check{\omega}_{1}$. 
\end{enumerate}
Note then that this iteration is appropriate, with each $\kappa_{\xi}$ being the least measurable cardinal above $|\bbP_{\xi}|$, each $\alpha_{\xi}$ being $\pi_{1}(\xi)$ and each $\dot{S}_{\xi}$ being a name built from the two possibilities above for $\dot{S}$ for successor $\xi$ (and $\check{\omega}_{1}$ when $\xi$ is $0$ or a limit ordinal).

Theorem \ref{mainthrm} then follows from the following theorem. 

\begin{thm}\label{mainthrmdetail}
Let $\langle {\mathbb P}_\eta , {\dot {\mathbb Q}}_\xi \colon \eta \leq \delta ,\, \xi < \delta\rangle$ be the iteration defined above, 
let $g$ be ${\mathbb P}_\delta$-generic over $V$, and let  $\vec{A}$ be the partition of $\omega_{1}$ induced by $g(0)$. Then the following hold in $V[g]$.
\begin{enumerate}

\item  For all $S \subseteq \omega_1$ and $\gamma < \delta$, $S$ is certified at $\gamma$ via $\vec{A}$ if and only if there is a successor $\xi < \delta$ such that $\dot{S}_{\xi, g_{\xi}} \subseteq S$ and $\gamma = \kappa_{\xi}$. 

\item For all $S \subseteq \omega_{1}$, $S$ is stationary if and only if there is a $\gamma < \delta$ such that  $S$ is certified at $\gamma$ via $\vec{A}$. 

\item ${\sf PFA} +  \neg \CFB$ $+$ ``${\sf NS}_{\omega_1}$ is $\Pi_1$-definable in a parameter from $H_{\aleph_{2}}$".
\end{enumerate}
\end{thm}

\begin{proof}
The reverse direction of part (1) follows from Lemma \ref{semi-proper}, and the fact that being certified at some $\gamma$ via some $\vec{A}$ is upwards absolute. For the forward direction of (1), fix 
\begin{itemize}
	\item $\eta < \delta$
	\item a $\bbP_{\delta}$-name $\dot{S}$ for a subset of $\omega_{1}$,
	\item $\gamma < \delta$,
	\item a condition $p \in \bbP_{\delta}$, 
	\item a countable ordinal $\beta$ which $p$ forces not to be in $\dot{S}$,   
	\item a $\bbP_{\delta}$-name $\dot{C}$ for a club subset of $\omega_{1}$ and 
	\item a $\bbP_{\delta}$-name $\dot{b}$ for a bijection between $\omega_{1}$ and $\gamma$. 
\end{itemize}. 
Let $\theta$ be a regular cardinal greater than $2^{|\bbP_{\delta}|}$ and let $X$ be a countable elementary submodel of $H_{\theta}$ with all the objects named above in $X$. Every $(X, \bbP_{\delta})$-generic condition forces that $\omega_{1} \cap X$ is in $\dot{C}$ and $\dot{b}[\omega_{1} \cap X] = X \cap \gamma$. It suffices to find, under the assumption that either (Case 1) $\gamma$ is not equal to any $\kappa_{\xi}$ or (Case 2) $\gamma$ is equal to some $\kappa_{\xi}$ and $p$ forces $\beta$ to be in $\dot{S}_{\xi}$, an $(X, \bbP_{\delta})$-generic condition $r \leq p$ forcing that $g(0)(\otp(X \cap \gamma)) \not\in \dot{S}$. Since the proofs of the two cases are similar we do them simultaneously. 

Fix $q <_{{\mathbb P}_0} p(0)$ such that
\begin{enumerate}
	\item $q$ is $(X,{\mathbb P}_0)$-generic,
	\item ${\rm dom}(q) = \otp(X \cap \delta)$,
	\item for each $\xi \in X \cap \delta$, if $\kappa_{\xi} \neq \gamma$ then $q(\otp(X \cap \kappa_{\xi})) = \alpha_{\xi}$ and 
	\item  $q({\rm otp}(X \cap \gamma))=\beta$.  
\end{enumerate}
Construing $p$ as $(p(0), \dot{p})$, we have that hypotheses (a) through (d) of Lemma \ref{lemma} are satisfied.
If $r$ is then as being given by the conclusion of Lemma \ref{lemma}, then $r$ is as desired. 

The reverse direction of part (2) follows from part (1) and the construction, and the fact that the tails of the iteration are semi-proper. 
(see Remark \ref{sprem}). 
The forward direction is by our bookkeeping. For each stationary $S \subseteq \omega_{1}$ in $V[g]$ there exist $\eta_{*} < \delta$ and $\alpha_{*} < \omega_{1}$ (any member of $S$) such that $S$ is the $g_{\eta}$-realization of some element $\dot{S}$ of $N_{\eta_{*}, \alpha_{*}}$. 
Working by induction, it suffices to suppose that the realization of each member of $N_{\eta_{*}, \alpha_{*}}$ which is $\leq_{\eta_{*}}$-below $\dot{S}$ is certified in $V[g]$. We may then let $\rho \in [\xi_{*}, \delta)$ be such that all of these certifying sets exist in $V[g \restrict \bbP_{\rho}]$. For cofinally many successor ordinals $\xi < \delta$, $\pi(\xi) = (\eta_{*}, \alpha_{*})$. For any such $\xi \geq \rho$, we have by Lemma \ref{semi-proper2} that $S$ is certified at $\kappa_{\xi}$ in $V[g \restrict (\xi + 1)]$. 

For part (3) that $V[g]$ is a model of ${\sf PFA}$ follows by the standard consistency proof for ${\sf PFA}$. The boldface $\Pi_1^{H_{\aleph_2}}$-definability of 
${\sf NS}_{\omega_1}$ in $V[g]$
follows from (2), the parameter being $g(0)$ (or the partition of $\omega_{1}$ induced by $g(0)$). The failure of $\CFB$ follows by adding $q(X \cap \omega_{1}) > \otp(X \cap \delta)$ to the argument for the forward direction of (1) (we may assume that $\gamma \geq \omega_{2}$ there, since only club subsets of $\omega_{1}$ have $\omega_{1}$ in their tildes).
\end{proof}

The argument provided here isn't tied to the nonstationary ideal. We could code any subset of ${\mathcal P}(\omega_1)$ in the fashion above, as long as the subset is closed under supersets and has a definition which is absolute to stationary set-preserving extensions.

\noindent Institut f\"ur Mathematische Logik, Universit\"at M\"unster, Einsteinstr. 62, FRG.

\noindent\text{stefan.hoffelner@gmx.at}

\bls
\noindent Department of Mathematics, Miami University, Oxford, Ohio 45056

\noindent \text{larsonpb@miamioh.edu}

\bls

\noindent Institut f\"ur Mathematische Logik, Universit\"at M\"unster, Einsteinstr. 62, FRG.

\noindent rds@uni-muenster.de

\bls

\noindent Institute of Mathematics, Academy of Mathematics and Systems Science,\\ Chinese Academy of Sciences, Beijing 100190, China

\smallskip

\noindent School of Mathematical Sciences, University of Chinese Academy of Sciences, Beijing 100049, China 

\noindent lzwu@math.ac.cn

\noindent\text{}

\begin{thebibliography}{10}
\bibitem{FMS} M. Foreman, M. Magidor, S. Shelah, \emph{Martin's Maximum, Saturated Ideals, and Non-Regular Ultrafilters. Part I.}
Annals of Mathematics 127 (1) 1988, 1-47
\bibitem{Fuchs} U. Fuchs, \emph{Donder’s version of revised countable support}, 2008,
arxiv.org/pdf/math/9207204v1.pdf.
\bibitem{HLSW} S.\ Hoffelner, P.\ Larson, R.\ Schindler, and L.\ Wu, {\em Forcing axioms and the definability of ${\sf NS}_{\omega_1}$}, Journal of Symbolic Logic, to appear
\bibitem{jech} T.\ Jech, {\bf Set theory}, Springer 2003
\bibitem{Kunen} K. Kunen, {\bf Set Theory: An Introduction to Independence Proofs}, North Holland, 1980
\bibitem{size} P.\ Larson, \emph{The size of $\tilde{T}$}, 
Archive for Mathematical Logic 39 (2000), 541–568 
\bibitem{stat-tower-book} P.\ Larson, {\bf The stationary tower}. University Lecture Series 32. American Mathematical Society. 2004
\bibitem{Laver} R. Laver, \emph{Making the supercompactness of $\kappa$ indestructible under $\kappa$-directed closed forcing}. Israel Journal of Mathematics. 29 (4) 1978, 385–388
\bibitem{LSS} Ph.\ Lücke, Ph.\ Schlicht, and R.\ Schindler, {\em Lightface definable subsets of $H_{\omega_2}$}, Journal Symb.\ Logic 82 (3) September 2017, 1106-1131.
\bibitem{Miyamoto} T. Miyamoto,
\emph{A limit stage construction for iterating semiproper preorders}. Proceedings of the 7th and 8th Asian Logic Conferences, 303–327.
\bibitem{SS} R. Schindler, X.\ Sun, in preparation.
\bibitem{Shelah} S. Shelah, {\bf Proper and Improper Forcing}. Perspectives in Mathematical
Logic, Springer-Verlag, Berlin, 1998
\bibitem{hugh} W.H.\ Woodin, {\bf The axiom of determinacy, forcing axioms, and
the nonstationary ideal}, de Gruyter 1999.
\end{thebibliography}
\end{document}